\documentclass[11pt]{amsart}

\usepackage{amsmath}
\usepackage{amsfonts}
\usepackage{amssymb}
\usepackage{amsthm}
\usepackage{tikz}
\usepackage{tikz-qtree}
\usepackage{mathdots}
\usepackage{algorithm,algpseudocode}
\usepackage{float}
\usepackage{caption}
\newtheorem{property}{Property}
\newtheorem{lemma}{Lemma}
\newtheorem{theorem}{Theorem}
\newtheorem{proposition}{Proposition}
\newtheorem{corollary}{Corollary}

\DeclareMathOperator{\lcm}{lcm}

\begin{document}

\markboth{S. Han et al}
{The $(u,v)$-Calkin-Wilf Forest}

%
%

\title{THE $(u,v)$-CALKIN-WILF FOREST}

\author{Sandie Han, Ariane M. Masuda, Satyanand Singh, and Johann Thiel}








\maketitle


\begin{abstract}
In this paper we consider a refinement, due to Nathanson, of the Calkin-Wilf tree. In particular, we study the properties of such trees associated with the matrices $L_u=\begin{bmatrix} 1 & 0 \\ u & 1\end{bmatrix}$ and $R_v=\begin{bmatrix} 1 & v \\ 0& 1\end{bmatrix}$, where $u$ and $v$ are nonnegative integers. We extend several known results of the original Calkin-Wilf tree, including the symmetry, numerator-denominator, and successor formulas, to this new setting. Additionally, we study the ancestry of a rational number appearing in a generalized Calkin-Wilf tree.
\end{abstract}



\section{Introduction}\label{intro}

The Calkin-Wilf tree~\cite{CW} is an infinite binary tree generated by two rules. The number 1, represented as $1/1$, is the root of the tree and each vertex $a/b$ has two children: the left one is $a/(a+b)$ and the right one is
$(a+b)/b$ (see Figure~\ref{fig:CWtree}).

\begin{figure}[ht!]
\begin{center}
\begin{tikzpicture}[sibling distance=15pt]
\tikzset{level distance=30pt}
\Tree[.$1/1$ [.$1/2$ [.$1/3$ $1/4$ $4/3$ ]
   [.$3/2$ $3/5$ $5/2$ ] ] [.$2/1$ [.$2/3$  $2/5$ $5/3$ ] [.$3/1$ $3/4$ $4/1$ ] ]]
\end{tikzpicture}
\end{center}
\caption{The first four rows of the Calkin-Wilf tree.}\label{fig:CWtree}
\end{figure}
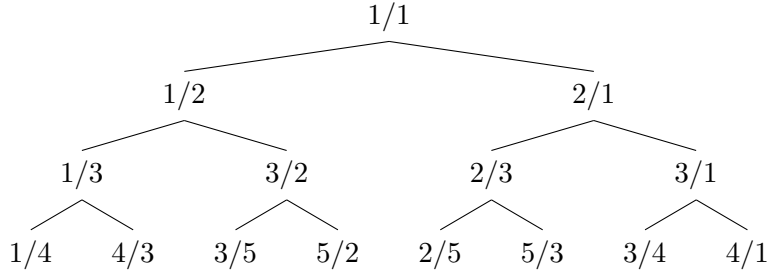

By following the breadth-first order, this tree provides an enumeration of positive rational numbers:
\[1, \dfrac{1}{2}, \dfrac{2}{1}, \dfrac{1}{3}, \dfrac{3}{2}, \dfrac{2}{3}, \dfrac{3}{1}, \dfrac{1}{4},  \dfrac{4}{3}
, \dfrac{3}{5}, \dfrac{5}{2}, \dfrac{2}{5}, \dfrac{5}{3} , \dfrac{3}{4} , \dfrac{4}{1},\ldots.\]
In fact, Calkin and Wilf~\cite{CW} showed that every reduced positive rational number appears in this list exactly once.

In addition to enumerating the positive rationals, the Calkin-Wilf tree has many interesting properties and generalizations that have been explored by various researchers (for example,~\cite{qCW, CW, DR, F, G, K, MS, N3, N2, N, SSS}). In particular, as in \cite{N3}, we highlight the following four properties. We denote by $c(n,i)$ the vertex in the $i^{\text{th}}$ position (from left to right) of the $n^\text{th}$ row\footnote{Our convention is that the row containing the root is the zero row. So, for example, $c(2,3)=2/3$.}.

\begin{property}
[Successor formula, Newman~\cite{N}] \label{sf}  For every nonnegative integer $n$ and $i=1,\ldots, 2^n-1$, we have
\begin{align}
c(n,i+1) &= \dfrac{1}{2[c(n,i)] + 1 - c(n,i)}\label{sfform}
\end{align}
where $[x]$ denotes the integer part of $x$.
\end{property}

\begin{property}
[Denominator-numerator formula, Calkin and Wilf~\cite{CW}] \label{dn} For every nonnegative integer $n$ and $i=1,\ldots, 2^n-1$, the denominator of $c(n,i)$ is equal to the numerator of $c(n,i+1)$.
\end{property}

\begin{property}
[Symmetry formula, \cite{N3}] \label{sn}  For every nonnegative integer $n$ and
$i=1,\ldots, 2^n$, we have
$c(n,i)\cdot c(n, 2^n-i+1) =1.$
\end{property}

\begin{property}
[Depth formula, \cite{GLB}] \label{df} Let $a/b$ be a positive reduced rational number. Let $n$ and $i$ be the unique pair such that $c(n,i)=a/b$ . If
\[
\frac{a}{b} =  a_0+\cfrac{1}{a_1+\ddots+\cfrac{1}{a_{k-1}+\cfrac{1}{a_k}}}
= [a_0,a_1,\ldots,a_{k-1},a_k]
\]
is the finite continued fraction representation\footnote{For a rational number not equal to 1, we always take the shorter continued fraction representation where $a_k\neq 1$.} of $a/b$, then \[n=a_0+a_1+\cdots+a_{k-1}+a_k-1.\]
In other words, the sum of the coefficients of the continued fraction representation encodes the row number where $a/b$ appears, i.e. the depth, in the Calkin-Wilf tree.

\end{property}

Let $z$ be a variable. In~\cite{N3}, Nathanson considers the infinite binary tree $\mathcal{T}(z)$, whose root is $z$, where each vertex $w$ has two children: the left child is $w/(w+1)$, and the right child is $w+1$ (see Figure~\ref{fig:Ntree}).

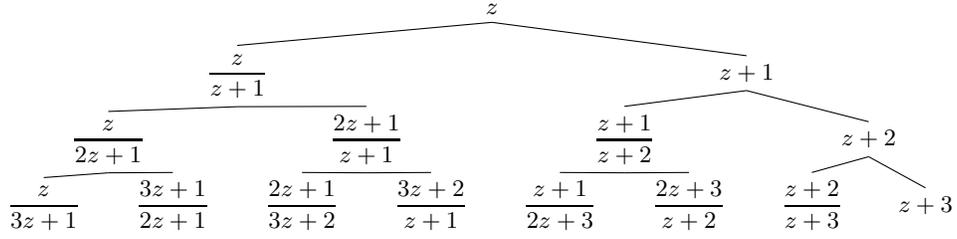
\begin{figure}[ht!]
\begin{center}
\begin{tikzpicture}[sibling distance=15pt]
\tikzset{level distance=25pt}
{\footnotesize{\Tree[.$z$ [.$\dfrac{z}{z+1}$ [.$\dfrac{z}{2z+1}$ $\dfrac{z}{3z+1}$ $\dfrac{3z+1}{2z+1}$ ]
   [.$\dfrac{2z+1}{z+1}$ $\dfrac{2z+1}{3z+2}$ $\dfrac{3z+2}{z+1}$ ] ] [.$z+1$ [.$\dfrac{z+1}{z+2}$  $\dfrac{z+1}{2z+3}$ $\dfrac{2z+3}{z+2}$ ]
    [.$z+2$ $\dfrac{z+2}{z+3}$ $z+3$ ] ]]}}
    \end{tikzpicture}
\end{center}
\caption{The first four rows of $\mathcal{T}(z)$.}\label{fig:Ntree}
\end{figure}

The original Calkin-Wilf tree is clearly the special case of $z=1$. For general $z$, Properties~\ref{sf}-\ref{df} of the Calkin-Wilf tree extend\footnote{Of independent interest, the generalization of Property~\ref{df} requires an appropriate definition of a continued fraction representation for linear fractional transformations.} to $\mathcal{T}(z)$.

We can associate each vertex in $\mathcal{T}(z)$ with a column vector as in Figure~\ref{fig:asso}.

\begin{figure}[ht!]
\begin{center}
\begin{tikzpicture}
\begin{scope}[xshift=-3cm,sibling distance=25pt]
\tikzset{level distance=40pt}
\Tree [.$z$ $\dfrac{z}{z+1}$ [.$z+1$  ] ]
\end{scope}
$\Longleftrightarrow$
\begin{scope}[xshift=3cm,sibling distance=25pt]
\tikzset{level distance=50pt}
\Tree [.${\begin{bmatrix} z\\1 \end{bmatrix}}$ ${\begin{bmatrix} z\\z+1 \end{bmatrix}}$ [.${\begin{bmatrix} z+1\\1 \end{bmatrix}}$  ] ]
\end{scope}
\end{tikzpicture}
\end{center}
\caption{Association between rational numbers and vectors.}\label{fig:asso}
\end{figure}
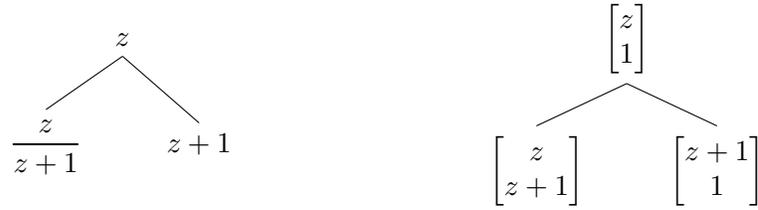

Letting $L_1:=\begin{bmatrix} 1 & 0\\ 1 & 1 \end{bmatrix}$ and $R_1:=\begin{bmatrix} 1 & 1\\ 0 & 1 \end{bmatrix}$, we see that subsequent vertices in $\mathcal{T}(z)$ can be obtained by matrix multiplication.  A vertex $\begin{bmatrix} a \\ b \end{bmatrix}$
has left child
\begin{align}
L_1\cdot\begin{bmatrix} a \\ b \end{bmatrix} = \begin{bmatrix} 1 & 0\\ 1 & 1 \end{bmatrix}\cdot\begin{bmatrix} a\\ b \end{bmatrix}
= \begin{bmatrix} a \\ a+b \end{bmatrix}\label{l1}
\end{align}
and right child
\begin{align}
R_1\cdot\begin{bmatrix} a \\ b \end{bmatrix} = \begin{bmatrix} 1 & 1\\ 0 & 1 \end{bmatrix}\cdot\begin{bmatrix} a \\ b \end{bmatrix}
= \begin{bmatrix} a+b \\ b \end{bmatrix}.\label{r1}
\end{align}
In particular, every vertex in $\mathcal{T}(z)$ is obtained by multiplying a matrix generated freely by the set $\{L_1, R_1\}$ with the vector associated with $z$. In this way, we can label the vertices of $\mathcal T(z)$ with matrices in $$SL_2(\mathbb{N}_0)=\bigg\{\begin{bmatrix} a & b\\ c & d \end{bmatrix}:a,b,c,d\in\mathbb{N}_0\text{ and }ad-bc=1\bigg\}$$ acting on $z$ (see Figure~\ref{fig:N2tree}). For ease of notation, we denote the left child and the right child of $w$ by $L_1(w)$ and $R_1(w)$, respectively.

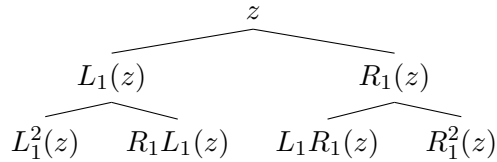
\begin{figure}[ht!]
\begin{center}
\begin{tikzpicture}[sibling distance=10pt]
\tikzset{level distance=25pt}
 \Tree [.$z$  [.$L_1(z)$  [.$L_1^2(z)$ ] [.$R_1L_1(z)$ ] ]
                [.$R_1(z)$  [.$L_1R_1(z)$   ]
                    [.$R_1^2(z)$    ]]]
 \end{tikzpicture}
 \end{center}
 \caption{The first three rows of $\mathcal{T}(z)$ in terms of $L_1$ and $R_1$.}\label{fig:N2tree}
\end{figure}

With this perspective in mind, it is natural to consider an analogous infinite binary tree generated by other pairs of matrices in $SL_2(\mathbb N_0)$. Let $u$ and $v$ be integers such that $u,v\ge 2$,
\[L_u:=\begin{bmatrix} 1 & 0 \\ u & 1\end{bmatrix} \qquad\text{ and }\qquad R_v:=\begin{bmatrix} 1 & v \\ 0 & 1\end{bmatrix}.\]
Nathanson \cite{M,N3} proposed to investigate the infinite binary tree associated to $\{L_u, R_v\}$ obtained by replacing $L_1$ in \eqref{l1} and $R_1$ in \eqref{r1} by $L_u$ and $R_v$, respectively (see Figure~\ref{fig:newCWrule} for the generation rule).

\begin{figure}[ht!]
\begin{center}
\begin{tikzpicture}
\begin{scope}[xshift=3cm,sibling distance=25pt]
\tikzset{level distance=40pt}
\Tree [.$w$ $\dfrac{w}{uw+1}$ [.$w+v$  ] ]
\end{scope}
\end{tikzpicture}
\end{center}
\caption{The children of $w$ in  $\mathcal{T}^{(u,v)}(z)$.}\label{fig:newCWrule}
\end{figure}
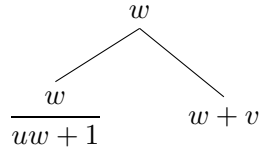

We refer to this generalization as a {\emph{$(u,v)$-Calkin-Wilf tree}} and denote it by $\mathcal{T}^{\;(u,v)}(z)$, where $z$ is the root (see Figure~\ref{fig:uvgraph}). Note that by setting $u=v=1$ and $z=1$, we obtain the original Calkin-Wilf tree, $\mathcal{T}(1)$. From now on, we assume that $u$ and $v$ are integers such that $u,v\ge 1$, and so $\mathcal{T}^{\;(1,1)}(1)$ is $\mathcal{T}(1)$.

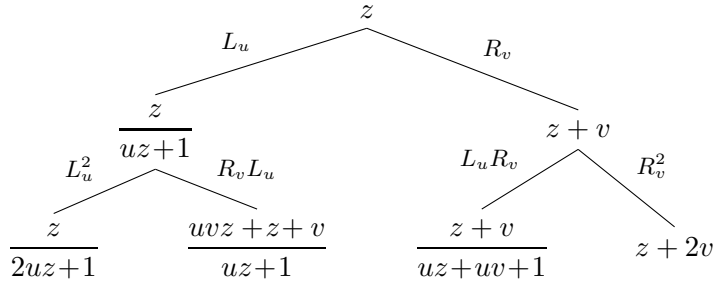
\begin{figure}[ht!]
\begin{center}
\begin{tikzpicture}[sibling distance=25pt]
\tikzset{level distance=45pt}
\Tree[.$z$ \edge node[auto=right] {\footnotesize{$L_u$}}; [.$\dfrac{z}{uz\!+\!1}$ \edge node[auto=right] {\footnotesize{$L^2_u$}}; [.$\dfrac{z}{2uz\!+\!1}$ ] \edge node[auto=left] {\footnotesize{$R_vL_u$}};
[.$\dfrac{uvz+\!z\!+v}{uz\!+\!1}$ ] ] \edge node[auto=left] {\footnotesize{$R_v$}}; [.$z+v$ \edge node[auto=right] {\footnotesize{$L_uR_v$}}; [.$\dfrac{z+v}{uz\!+\!uv\!+\!1}$ ] \edge node[auto=left] {\footnotesize{$R^2_v$}}; [.$z+2v$ ] ] ]
\end{tikzpicture}
\end{center}
\caption{The first three rows of $\mathcal T^{\;(u,v)}(z)$.}\label{fig:uvgraph}
\end{figure}

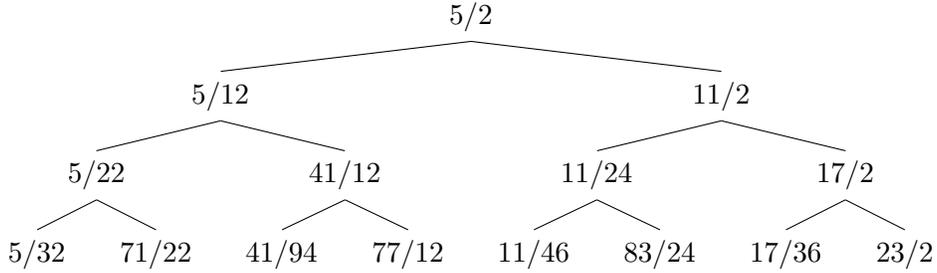
\begin{figure}[ht!]
\begin{center}
\begin{tikzpicture}[sibling distance=13pt]
\tikzset{level distance=30pt}
\Tree[.$5/2$ [.$5/12$ [.$5/22$ $5/32$ $71/22$ ]
   [.$41/12$ $41/94$ $77/12$ ] ] [.$11/2$ [.$11/24$ $11/46$ $83/24$ ] [.$17/2$ $17/36$ $23/2$ ]  ] ]
\end{tikzpicture}
\end{center}
\caption{The first four rows of $\mathcal T^{\;(2,3)}(5/2)$.}\label{fig:u2v3}
\end{figure}

As an example, consider the tree $\mathcal T^{\;(2,3)}(5/2)$ (see Figure~\ref{fig:u2v3}). One can immediately notice that the denominator-numerator and the symmetry formulas (Properties~\ref{dn} and~\ref{sn}) do not hold in $\mathcal T^{\;(2,3)}(5/2)$. Furthermore, many rational numbers appearing in $\mathcal{T}(1)$ seem to be missing in $\mathcal T^{\;(2,3)}(5/2)$. In fact, it is not too difficult to show that 1 does not appear in any tree $\mathcal T^{\;(u,v)}(z)$ unless $z=1$. In the next section we will address this issue, and define the $(u,v)$-Calkin-Wilf forest which will enumerate positive rational numbers.

We have already shown by example that Properties~\ref{sf}-\ref{df} do not, in general, hold for a $(u,v)$-Calkin-Wilf tree. However, $(u,v)$-Calkin-Wilf trees share enough of a similar structure with $\mathcal{T}(1)$ that we are able to provide some appropriate, universal generalizations (see Theorem~\ref{gensym} and Corollary~\ref{cfdepth}, for example). In other cases, we will show that some of Properties~\ref{sf}-\ref{df} completely characterize the Calkin-Wilf tree (see Proposition~\ref{denufor} and Corollary~\ref{symfor}, for example).


\section{Global Properties}

For a fixed $u$ and $v$, consider the set of all positive reduced rational numbers that are not the children of any rational number appearing in \emph{any} $(u,v)$-Calkin-Wilf tree. We refer to such numbers as $(u,v)$-orphans (when the context is clear, we may refer to such numbers simply as orphans). A straightforward proof shows that the set of $(u,v)$-orphans is
\begin{equation*}\left\{\dfrac{a}{b}\colon \dfrac{1}{u}\le\dfrac{a}{b}\le v\right\}\end{equation*}
(see~\cite{N3}). It follows that the set of $(u,v)$-orphans is finite if and only if $u=v=1$. Furthermore, it can be seen that every left child in a $(u,v)$-Calkin-Wilf tree is strictly bounded above by $1/u$ and every right child is strictly bounded below by $v$. In the case of the original Calkin-Wilf tree $1$ is the only orphan. In $\mathcal T^{\;(2,3)}(5/2)$, the vertex $5/2$ satisfies the condition $1/2 \le 5/2 \le 3$, and so it is one of the many $(2,3)$-orphans.

\begin{lemma}\label{motlemma1}
Let $z$ and $z'$ be distinct $(u,v)$-orphans. Then the vertices of $\mathcal{T}^{(u,v)}(z)$ and $\mathcal{T}^{(u,v)}(z')$ form disjoint sets.
\end{lemma}
\begin{proof}
Suppose that $w$ is a rational number that appears as a vertex in both $\mathcal{T}^{(u,v)}(z)$ and $\mathcal{T}^{(u,v)}(z')$. Without loss of generality, we can assume that $w$ is such that no other ancestor of it (in either tree) holds this property. It follows $w$ is not a root and must be the child of vertices in both trees. Furthermore, $w$ cannot be a left child (right child, resp.) in both $\mathcal{T}^{(u,v)}(z)$ and $\mathcal{T}^{(u,v)}(z')$. So $w$ is a left child in, say, $\mathcal{T}^{(u,v)}(z)$ and a right child in $\mathcal{T}^{(u,v)}(z')$. This implies that $w<1/u\leq 1$ and $w>v\geq 1$, a contradiction.
\end{proof}

Since every positive reduced rational number is either a $(u,v)$-orphan or the descendant of a $(u,v)$-orphan, Lemma~\ref{motlemma1} shows that the set of $(u,v)$-orphans enumerates a forest of trees that partitions the set of positive rational numbers, the $(u,v)$-Calkin-Wilf forest.

\begin{lemma}\label{motlemma2}
Let $u$ and $v$ be positive integers.  Then $L_u=L_1^u$ and $R_v=R_1^v$.
\end{lemma}
\begin{proof}
We show that $L_u=L_1^u$ by induction on $u$.  This is clearly true when $u=1$.  Suppose it is true for $u\geq 1$.  Then
\[L_1^{u+1} = L_1^u\cdot L_1 = L_u\cdot L_1  =
\begin{bmatrix}
1 & 0\\
u  & 1\\
\end{bmatrix}\cdot
\begin{bmatrix}
1 & 0\\
1  & 1\\
\end{bmatrix} =
\begin{bmatrix}
1 & 0\\
u+1  & 1\\
\end{bmatrix}=
L_{u+1}. \]
A similar argument shows that $R_v=R_1^v$.
\end{proof}

When comparing $\mathcal{T}^{(u,v)}(z)$ to $\mathcal{T}(1)$, one can see from Lemma~\ref{motlemma2} that the vertices of $\mathcal{T}^{(u,v)}(z)$ can be obtained by starting with the vertex $z$ in $\mathcal{T}(1)$ and skipping over $u-1$ generations of left children or $v-1$ generations of right children to arrive at the children of $z$ in $\mathcal{T}^{(u,v)}(z)$. For example, compare Figure~\ref{fig:N2tree} and Figure~\ref{fig:matrix23tree} in the case where $u=2$ and $v=3$. In other words, the vertex set of $\mathcal{T}^{(u,v)}(z)$ is a submonoid  of $\mathbb Q$ or, equivalently, the vertex set of $\mathcal{T}(1)$. More generally, we have the following two results as other immediate consequences of Lemma~\ref{motlemma2}.

\begin{proposition}
The vertex set of $\mathcal T^{(u,v)}(z)$ is a submonoid of the vertex set of $\mathcal T^{(u',v')}(z')$ if and only if $z\in \mathcal T^{(u',v')}(z')$, $u'\mid u$, and $v'\mid v$.
\end{proposition}

\begin{proposition}
Let $U$ and $V$ be finite sets of nonnegative integers. Set $u :=\lcm\{u'\colon u'\in U\}$ and $v :=\lcm\{v'\colon v'\in V\}$.  Then \[\mathcal T^{(u,v)}(z) = \displaystyle\bigcap_{u'\in U, v'\in V} \mathcal T^{(u',v')}(z).\]
\end{proposition}

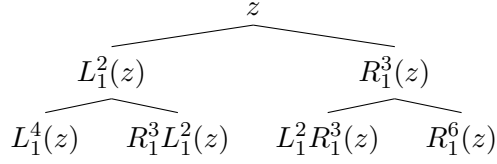
\begin{figure}
\begin{center}
\begin{tikzpicture}[sibling distance=10pt]
\tikzset{level distance=25pt}
 \Tree [.$z$  [.$L_1^2(z)$  [.$L_1^4(z)$ ] [.$R_1^3L_1^2(z)$ ] ]
                [.$R_1^3(z)$  [.$L_1^2R_1^3(z)$   ]
                    [.$R_1^{6}(z)$  ]]]
 \end{tikzpicture}
    \caption{The first three rows of $\mathcal{T}^{(2,3)}(z)$ in terms of $L_1$ and $R_1$.}\label{fig:matrix23tree}
 \end{center}
 \end{figure}

Lemma~\ref{motlemma1} and Lemma~\ref{motlemma2} show that the $(u,v)$-orphans partition the set of positive rational numbers into a collection of trees with a similar structure \emph{derived} from $\mathcal{T}(1)$. This idea serves as the main motivation for this paper.


\section{The Successor and the Numerator-Denominator Formulas}

 We begin this section by establishing some immediate properties of $(u,v)$-Calkin-Wilf trees related to Properties~\ref{sf}-\ref{dn} . We denote by $c_z^{(u,v)}(n,i)$ the $i^{\text{th}}$ element, from left to right, in the $n^{\text{th}}$ row of the $(u,v)$-Calkin-Wilf tree whose root is $z$. For any integer $n\ge 0$, the first and the last elements of the $n^{\text{th}}$ row with root $z$ are readily seen to be
 \[c_{z}^{(u,v)}(n,1) = \dfrac{z}{nuz+1}\quad \text{ and }\quad c_{z}^{(u,v)}(n,2^n) = z+nv,\]
respectively.  Furthermore, since $z$ is assumed to be in reduced form, then all vertices of $\mathcal{T}^{\;(u,v)}(z)$ are also in reduced form.

\begin{proposition}[Generalized successor formula]\label{sfuv}
Consider the $(u,v)$-Calkin-Wilf tree with root $z$. For every nonnegative integer $n$ and $i=1,\dots,2^n-1$, let $\alpha_i=c_{z}^{(u,v)}(n,i)$.  Then we have
\begin{align}
\alpha_{i+1} &=\frac{v\{\alpha_i\}+v^2(1-u\{\alpha_i\})}{u\lbrack \alpha_i\rbrack\big(\{\alpha_i\}+v(1-u\{\alpha_i\})\big)+v(1-u\{\alpha_i\})}\label{sfuvform}
\end{align}
where $\lbrack x\rbrack$ and $\{x\}$ denote the integer and fractional parts of the real number $x$, respectively.
\end{proposition}

\begin{proof}
Our proof is a generalization of an argument by Newman~\cite{AZ} in the case where $u=v=1$.

If $\alpha_i$ and $\alpha_{i+1}$ are adjacent siblings in a $(u,v)$-Calkin-Wilf tree, then they share a common ancestor $w$ (see Figure \ref{fig:suc}) such that, for some $k\geq 0$, $\alpha_i$ is the $k^{\text{th}}$ right child of the left child of $w$ and $\alpha_{i+1}$ is the $k^{\text{th}}$ left child of the right child of $w$. (This is not a feature that is unique to $(u,v)$-Calkin-Wilf trees; it is common to all full binary trees.) It follows that $\alpha_i=\frac{w}{uw+1}+kv$ and $\alpha_{i+1}=\frac{w+v}{ku(w+v)+1}$. Note that since $\frac{w}{uw+1}<1$, then $\{\alpha_i\}=\frac{w}{uw+1}$ and $[\alpha_i]=kv$.

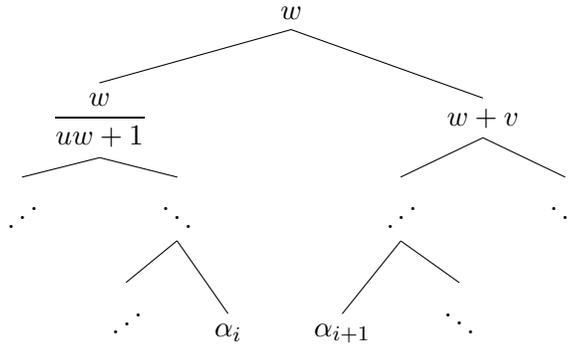
\begin{figure}[ht!]
\begin{center}
\begin{tikzpicture}[sibling distance=20pt]
\tikzset{level distance=40pt}
\Tree [.$w$ [.$\dfrac{w}{uw+1}$ $\iddots$  [.$\ddots$ $\iddots$ $\alpha_i$ ] ]
            [.$w+v$ [.$\iddots$ $\alpha_{i+1}$ $\ddots$ ] $\ddots$ ] ]
\end{tikzpicture}
\end{center}
\caption{Successors in a $(u,v)$-Calkin-Wilf tree with common ancestor $w$.}
\label{fig:suc}
\end{figure}

In order to complete the proof, we must eliminate the dependence of $\alpha_{i+1}$ on $k$ and $w$. This can be accomplished by taking the formula for $\{\alpha_i\}$ and solving for $w$. This gives that
\begin{align}\label{yinalpha}
w & = \frac{\{\alpha_i\}}{1-u\{\alpha_i\}}.
\end{align}
It follows that
\begin{align}\label{nextalpha}
\alpha_{i+1} = \frac{w+v}{ku(w+v)+1}
 = \frac{w+v}{kv(\frac{wu}{v}+u)+1}
 = \frac{w+v}{[\alpha_i](\frac{wu}{v}+u)+1}.
\end{align}

Inserting \eqref{yinalpha} into the right-hand side of \eqref{nextalpha} and simplifying gives the desired result.
\end{proof}

While \eqref{sfuvform} does collapse down to \eqref{sfform} when $u=v=1$, something is lost in this generalization. Iterating \eqref{sfform} not only gives successive elements in a fixed row of the Calkin-Wilf tree; when it is applied to the rightmost element of a row, it returns the leftmost element of the next row. The same is not true of \eqref{sfuvform}.

It follows from Proposition~\ref{sfuv} that if we consider successive terms in each row of a $(u,v)$-Calkin-Wilf tree, the denominator-numerator formula (Property~\ref{dn}) holds only in the original Calkin-Wilf tree.

\begin{proposition}\label{denufor}
The denominator-numerator formula holds if and only if $u=v=1$.
\end{proposition}
\begin{proof}
Using the same notation in the proof of Proposition \ref{sfuv}, for a common ancestor $w$,
\[\alpha_i=\frac{w'+kv(uw'+w'')}{uw'+w''} \quad\text{ and }\quad \alpha_{i+1}=\frac{w'+vw''}{ku(w'+vw'')+w''},\]
where $w=w'/w''$ is in lowest terms. It is easy to see that the above representations of $\alpha_i$ and $\alpha_{i+1}$ are also in lowest terms.
So we can let $d_i=uw'+w''$ be the denominator of $\alpha_i$ and $n_{i+1}=w'+vw''$ be the numerator of $\alpha_{i+1}$. It quickly follows that
\begin{align}
vd_i+(1-uv)w'=n_{i+1}.\label{dnfbad}
\end{align}

$(\Leftarrow)$ If $u=v=1$, then $d_i=n_{i+1}$ follows from \eqref{dnfbad}.

$(\Rightarrow)$ If $d_i=n_{i+1}$, then it follows from \eqref{dnfbad} that \[
(uv-1)w' =(v-1)n_{i+1}
=(v-1)(w'+vw'').\]
Collecting like terms on either side of the equality shows that $(u-1)w'=(v-1)w''$. If $u=1$ and $v\neq 1$, then $w''=0$, a contradiction. A similar argument works for the case where $u\neq 1$ and $v=1$. If $u,v\neq 1$, then $w=w'/w''=(v-1)/(u-1)$. This would imply that $w$ is fixed for all pairs of successors, another contradiction. Therefore, $u=v=1$.
\end{proof}

We see from \eqref{dnfbad} that the relationship between successive denominators and numerators in a row of a $(u,v)$-Calkin-Wilf tree is significantly more complicated than in the statement of Property~\ref{dn}. In order to generalize the denominator-numerator formula, one would need to know more about the common ancestors of successive terms. At this time, no clear generalization of Property~\ref{dn} is evident.


\section{Symmetry Properties}

In this section, we study symmetry properties of $(u,v)$-Calkin-Wilf trees closely related to Property~\ref{sn}. As in the previous section, we are able to find some appropriate generalizations, in some sense, while showing that Property~\ref{sn} completely characterizes $\mathcal{T}(1)$. We begin with a lemma which will be used in the theorems that follow.

\begin{lemma}\label{replemma}
For every vertex in the $(u,v)$-Calkin-Wilf tree with root $z$ there are nonnegative integers $a, b, c,$ and $d$ with $ad-bc=1$ such that the vertex is represented as $\dfrac{az+b}{cz+d}.$
\end{lemma}

\begin{proof}
The statement follows from induction on the row number of the $(u,v)$-Calkin-Wilf tree $\mathcal{T}^{(u,v)}(z)$ (see Figure~\ref{fig:uvgraph}).
\end{proof}

Note that the integers $a$, $b$, $c$, and $d$ in Lemma~\ref{replemma} depend on $u$, $v$, and the position of the vertex in the tree. See~\eqref{lem1example} in Section~\ref{dcdf} for an example on how to compute $a$, $b$, $c$, and $d$ for $2147/620$ in $\mathcal{T}^{(2,3)}(5/2)$. Furthermore, Lemma~\ref{replemma} shows that every vertex in a $(u,v)$-Calkin-Wilf tree can be written as some linear fractional transformation of the root (see Figures~\ref{fig:uvgraph} and~\ref{uvvu}).

\begin{theorem}[General symmetry formula]\label{gensym}
For every nonnegative integer $n$ and $i=1,2,\ldots, 2^n$,
if $c_{z}^{(u,v)}(n,i)=\dfrac{a z +b}{c z + d}$ where $a,b,c,d$ are nonnegative integers, then
\begin{equation}c_{z}^{(u,v)}(n,2^n+1-i)=\dfrac{d z +\dfrac{cv}{u}}{\dfrac{bu}{v}z + a}.\label{gensymform}\end{equation}
\end{theorem}
\begin{proof}
The proof is by induction on the row number $n$. Since $c_{z}^{(u,v)}(1,1)=\dfrac{z}{uz+1}$, we have that
$c_{z}^{(u,v)}(1,1) =\dfrac{a z +b}{c z + d}$ with $a =1$, $b = 0$, $c =u,$ and $d = 1$, and so
\[\dfrac{d z +\dfrac{cv}{u}}{\dfrac{bu}{v}z + a} = z+v=c_{z}^{(u,v)}(1,2).\]
On the other hand, starting from $c_{z}^{(u,v)}(1,2)=z+v$, we get that $c_{z}^{(u,v)}(1,2)=\dfrac{a z +b}{c z + d}$
with $a =1$, $b = v$, $c =0,$ and $d = 1$. Hence
\[\dfrac{d z +\dfrac{cv}{u}}{\dfrac{bu}{v}z + a} = \dfrac{z}{uz+1} = c_{z}^{(u,v)}(1,1).\]
This shows that the statement is true when $n=1$. Suppose that the theorem is true for some row $n\geq 1$. An element in the row $n+1$
is either of the form $c_{z}^{(u,v)}(n+1,2i-1)$ or $c_{z}^{(u,v)}(n+1,2i)$ for some integer $i$, $1\le i\le 2^{n}$.  If $c_{z}^{(u,v)}(n+1,2i-1)=\dfrac{a z+b}{c z + d}$ (we know that such a representation exists by Lemma~\ref{replemma}) then it is the left child of
\[c_{z}^{(u,v)}(n,i) = \dfrac{a z +b}{(c-ua)z+(d-ub)}.\]
Thus, by using the symmetry on row $n$, we obtain
\begin{eqnarray*}
&& c_{z}^{(u,v)}(n+1,2^{n+1}+2-2i) = R_v\left(c_{z}^{(u,v)}(n,2^n+1-i)\right) \\
&=& R_v\left(\dfrac{(d - u b)z + \dfrac{(c-ua)v}{u}}{\dfrac{bu}{v}z + a}\right)
= \dfrac{d z +\dfrac{cv}{u}}{\dfrac{bu}{v}z + a}.
\end{eqnarray*}
Similarly, if $c_{z}^{(u,v)}(n+1,2i)=\dfrac{a z +b}{c z + d}$ then it is the right child of
\[c_{z}^{(u,v)}(n,i) = \dfrac{(a-c v) z +(b-vd)}{c z+d}. \]
Hence
\begin{eqnarray*}
&&c_{z}^{(u,v)}(n+1,2^{n+1}+1-2i) = L_u\left(c_{z}^{(u,v)}(n,2^n+1-i)\right) \\
&=& L_u\left(\dfrac{d z + \dfrac{cv}{u}}{\dfrac{(b-vd)u}{v}z + (a-c v)}\right)
= \dfrac{d z +\dfrac{cv}{u}}{\dfrac{bu}{v}z + a}.
\end{eqnarray*}
\end{proof}

As a consequence, we obtain necessary and sufficient conditions for the symmetry formula (Property~\ref{sn}) to hold in a $(u,v)$-Calkin-Wilf tree.

\begin{corollary}[Symmetry formula]\label{symfor}
The symmetry formula, \begin{equation}c_{z}^{(u,v)}(n,i) \cdot c_{z}^{(u,v)}(n,2^n+1-i)=1,\label{eq1}\end{equation}  holds if and only if  $u=v$ and $z=1$.
\end{corollary}
\begin{proof}
Suppose, using Lemma~\ref{replemma}, that $c_{z}^{(u,v)}(n,i)=\dfrac{a z +b}{c z + d}$ where $a,b,c,d$ are nonnegative integers and $ad-bc=1$.  By Theorem~\ref{gensym}, we obtain that \eqref{eq1} is equivalent to
\[
\left(\dfrac{a z +b}{c z + d} \right)\cdot \left(\dfrac{d z +\dfrac{cv}{u}}{\dfrac{bu}{v}z + a}\right)= 1,\]
or
\begin{equation}\label{multipl}\left(ad- \dfrac{bcu}{v} \right)z^2+\left[bd\left(1-\frac{u}{v}\right)-ac\left(1-\dfrac{v}{u}\right)\right]z+\left(\dfrac{bcv}{u}-ad\right)=0
\end{equation}
It follows that $ad- \dfrac{bcu}{v}=0$ and $\dfrac{bcv}{u}-ad=0$, from which we get
\[\dfrac{v}{u}=\dfrac{bc}{ad} \qquad \text{and}\qquad \dfrac{v}{u}=\dfrac{ad}{bc},\]
thus $v^2=u^2$.  Since $u, v>0$, this implies that $u=v$. By substituting $u=v$ into (\ref{multipl}), we obtain that
$(ad-bc)(z^2-1)=0$.
Since $ad-bc=1$ and $z>0$, we conclude that $z=1$.
\end{proof}

We remark that Corollary~\ref{symfor} can be also proved using induction on the row number.
The result explains why the symmetry formula does not hold in $\mathcal T^{\;(2,3)}(5/2)$ (see Figure~\ref{fig:u2v3}), as we had observed earlier.

\begin{corollary}[Skew symmetry]\label{skewsym}
Using the same hypothesis as Theorem~\ref{gensym}, it follows that
\begin{equation*}c_{z}^{(u,v)}(n,i) \cdot c_{\frac{v}{u}z^{-1}}^{(u,v)}(n,2^n+1-i)=\frac{v}{u}.\end{equation*}
\end{corollary}

\begin{proof}
Suppose, using Lemma~\ref{replemma}, that $c_{z}^{(u,v)}(n,i)=\dfrac{a z +b}{c z + d}$ where $a,b,c,d$ are nonnegative integers. Replacing $z$ by $\frac{v}{uz}$ in \eqref{gensymform} yields that
\begin{equation*}
c_{\frac{v}{u}z^{-1}}^{(u,v)}(n,2^n+1-i) =\dfrac{d \dfrac{v}{uz} +\dfrac{cv}{u}}{\dfrac{bu}{v}\dfrac{v}{uz} + a}
=\dfrac{\dfrac{v}{u}\big(d+cz\big)}{b + az}
=\dfrac{v}{u}\cdot\dfrac{cz+d}{az+b},
\end{equation*}
which is equivalent to the desired result.
\end{proof}

Corollary~\ref{symfor} shows that the symmetry formula does not hold for $(u,v)$-Calkin-Wilf trees in general. However, Corollary~\ref{skewsym} (above) and Theorem~\ref{n2} (below) show that other symmetry formulas do hold when comparing either pairs of $(u,v)$-Calkin-Wilf trees or $(u,v)$- and $(v,u)$-Calkin-Wilf trees, respectively. For examples, see Table~\ref{table1}.

\begin{table}[ht!]
\begin{center}
\begin{tabular}{|c|cccc|}
\hline
Row 2 of $\mathcal{T}^{(2,3)}(5/2)$ & $5/22$ & $41/12$ & $11/24$ & $17/2$\\
\hline
Row 2 of $\mathcal{T}^{(2,3)}(3/5)$ & $3/17$ & $36/11$ & $18/41$ & $33/5$\\
\hline
Row 2 of $\mathcal{T}^{(3,2)}(2/5)$ & $2/17$ & $24/11$ & $12/41$ & $22/5$\\
\hline
\end{tabular}
\caption{Examples of Corollary~\ref{skewsym} and Theorem~\ref{n2}}\label{table1}
\end{center}
\end{table}

\begin{theorem} [Nathanson's symmetry,~\cite{N2}] \label{n2}
Let $z$ be a variable, and let $u$ and $v$ be positive integers.
For all nonnegative integers $n$ and $i=1, 2, \ldots, 2^n$,
\[c_{z}^{(u,v)}(n,i)\cdot c_{z^{-1}}^{(v,u)}(n, 2^n + 1 - i) = 1.\]
\end{theorem}

If $u=v\geq 1$, then Theorem~\ref{n2} gives one of the directions of Corollary~\ref{symfor}. If $u = v = 1$, then this is the familiar symmetry of the Calkin-Wilf tree.

Nathanson's symmetry was proved in~\cite{N2} using induction on the row number. We conclude this section with two alternative proofs of Theorem~\ref{n2}. The first one is a consequence of Theorem~\ref{gensym}, and only holds when $u=v$.

\begin{proof}[First Proof of Theorem~\ref{n2} when $u=v$]
By Lemma~\ref{replemma}, let $c_z^{(u,u)}(n,i) = \dfrac{a z +b}{c z + d}$ for some nonnegative integers $a, b, c,$ and $d$. By Theorem~\ref{gensym},
we have
\[c_z^{(u,u)}(n,2^n+1-i)=\dfrac{d z +c}{bz + a}
\Longrightarrow c_{z^{-1}}^{(u,u)}(n,2^n+1-i)=\dfrac{d z^{-1} +c}{bz^{-1} + a} = \dfrac{cz+d}{az+b},\]
which is the reciprocal of $c_z^{(u,u)}(n,i)$.
\end{proof}

The identity presented in Theorem~\ref{gensym} only holds in a $(u,v)$-Calkin-Wilf tree where $u$ and $v$ are fixed.
Therefore we cannot use it to derive Nathanson's symmetry in the case $u\ne v$.  In order to show the desired relationship
between $(u,v)$- and $(v,u)$-Calkin-Wilf trees, we will use a lemma that shows the following:

\begin{lemma}\label{composition}
Let $\sigma : \mathbb{Q^*}\to \mathbb{Q^*}$ be defined by $\sigma(x)=x^{-1}$. Then
\begin{itemize}
\item[(a)] $\sigma \circ L_u \circ \sigma = R_u$
\item[(b)] $\sigma \circ R_u \circ \sigma = L_u$
\end{itemize}
\end{lemma}
\begin{proof}
Part (a) of the lemma follows from the following straightforward computation:
    \begin{equation*}
    (\sigma \circ L_u \circ \sigma)(x) = \sigma \left(\dfrac{x^{-1}}{ux^{-1}+1}\right)
     = \sigma\left(\dfrac{1}{u+x}\right)
     =x+u
     = R_u(x).
    \end{equation*}
Part (b) follows from (a) since $\sigma^2=id$
\end{proof}

\begin{figure}[ht!]
\begin{center}
\begin{tikzpicture}[sibling distance=25pt]
\tikzset{level distance=45pt}
\Tree[.$z^{-1}=\frac{1}{z}$ \edge node[auto=right] {\footnotesize{$L_v$}}; [.$\dfrac{1}{z+v}$ \edge node[auto=right] {\footnotesize{$L^2_v$}}; [.$\dfrac{1}{z+2v}$ ] \edge node[auto=left] {\footnotesize{$R_uL_v$}};
[.$\dfrac{uz+uv+\!1}{z\!+\!v}$ ] ] \edge node[auto=left] {\footnotesize{$R_u$}}; [.$\dfrac{uz+1}{z}$ \edge node[auto=right] {\footnotesize{$L_vR_u$}}; [.$\dfrac{uz+1}{uvz\!+\!z\!+\!v}$ ] \edge node[auto=left] {\footnotesize{$R^2_u$}}; [.$\dfrac{2uz+1}{z}$ ] ] ]
\end{tikzpicture}
    \caption{The first three rows of $\mathcal{T}^{(v,u)}(z^{-1})$.}\label{uvvu}
 \end{center}
 \end{figure}
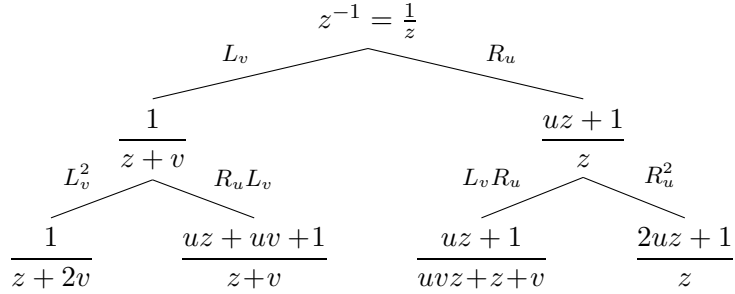

Comparing Figures~\ref{fig:uvgraph} and~\ref{uvvu}, we can see that if we view $c_{z}^{(u,v)}(n,i)$ as the result of a (unique) word $w(L_u,R_v)$ on two letters acting on $z$, then  $c_{z}^{(v,u)}(n,2^n+1-i)= w(R_u,L_v)(z)$. Specifically, the vertex $c_{z}^{(u,v)}(n,i)$ is $w(L_u,R_v)(z)$ where $w$ is
the $i^\text{th}$ word of length $n$ on the letters $R_u$ and $L_v$ in the reverse lexicographic order. We will use this approach to prove Nathanson's symmetry in its general form.

\begin{proof}[Second Proof of Theorem~\ref{n2}]
Let $c_z^{(u,v)}(n,i)=w(L_u,R_v)(z)$ where $w$ is the $i^\text{th}$ word of length $n$ on the letters $R_u$ and $L_v$ in the reverse lexicographic order.
For $\sigma(z) = z^{-1}$, it follows from Lemma~\ref{composition} that
\[\sigma\circ w(L_u,R_v)\circ\sigma = w(\sigma\circ L_u\circ\sigma,\sigma\circ R_v\circ\sigma) = w( R_u, L_v).\]
Therefore  $\sigma\circ w(L_u,R_v) = w( R_u, L_v)\circ\sigma$, which means that $\left(c_z^{(u,v)}(n,i)\right)^{-1}= c_{z^{-1}}^{(v,u)}(n,2^n+1-i)$.
\end{proof}


\section{The Descendant Conditions and the Depth Formula}\label{dcdf}

In a full binary tree, each vertex can be assigned a binary representation by enumerating the vertices in a breadth-first order. For example, the root of the tree is assigned the number 1; its left child is 2 and right child is 3, or $10_2$ and $11_2$ in their respective binary representations. In the next row, the vertices are 4, 5, 6, 7, or $100_2$, $101_2$, $110_2$, $111_2$, in binary representation form (See Figure~\ref{fig:bintree}).

\begin{figure}[ht!]
\begin{center}
\begin{tikzpicture}[sibling distance=3pt]
\tikzset{level distance=25pt}
5 \Tree [.$1_2$ [.$10_2$  [.$100_2$   $1000_2$   $1001_2$ ] [.$101_2$    $1010_2$   $1011_2$ ] ]
               [.$11_2$  [.$110_2$ $1100_2$   $1101_2$ ]
                    [.$111_2$   $1110_2$   $1111_2$
                          ] ] ]\end{tikzpicture}
\caption{Binary representation tree.}
   \label{fig:bintree}
\end{center}
\end{figure}
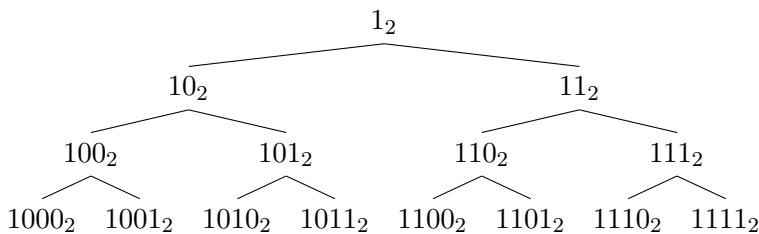

The parent-child relation is clearly demonstrated by the binary representation. Each left child is represented by the binary representation of its parent followed by a 0, while each right child is represented by the binary representation of its parent followed by a 1. Moreover, for each vertex, its binary representation encodes the binary representations of all of its ancestors back to the root.

We construct a 1-1 correspondence between the binary representations of the vertices in a full-binary tree and the words associated with each vertex in the Calkin-Wilf tree. Begin with the binary representation of a vertex. Truncate the leftmost 1 digit (all such representations begin with a 1), reverse the order of the string and map $0\mapsto L_1$ and $1\mapsto R_1$. For example, the vertex in position $1100_2$ corresponds to the word $L_1^2R_1$, which corresponds to the number $2/5=L_1^2R_1(1)$ in the Calkin-Wilf tree.

In the $(u,v)$-Calkin-Wilf tree, if we use the same binary representation as those in the original Calkin-Wilf tree, we can easily see that the left child is represented by the binary representation of its parent followed by $u$ consecutive 0s and the right child is represented by the binary representation of its parent followed by $v$ consecutive 1s. Let $B$ be the binary representation of the position of $w$ in the original Calkin-Wilf tree. Figure~\ref{fig:btree} shows the first three rows of $\mathcal{T}^{(2,3)}(w)$ in binary form.

\begin{figure}[ht!]
\begin{center}
\begin{tikzpicture}[sibling distance=3pt]
\tikzset{level distance=25pt}
5 \Tree [.$B$ [.$B00_2$  [.$B0000_2$  ] [.$B00111_2$  ] ]
               [.$B111_2$  [.$B11100_2$ ]
                    [.$B111111_2$
                          ] ] ]\end{tikzpicture}
\caption{Binary representation tree for $\mathcal{T}^{(2,3)}(w)$.}
   \label{fig:btree}
\end{center}
\end{figure}
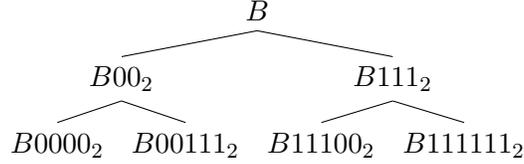

The $(u,v)$-ancestor-descendant relation is clearly demonstrated by the sequence of $u$ consecutive 0s or $v$ consecutive 1s. We give a few examples related to $(2,3)$-Calkin-Wilf trees:
\begin{itemize}
\item We have that $2/5 \mapsto 1100_2$, which is the left child of $11_2\mapsto 2$. Incidentally, 2 is an orphan root in the $(2,3)$-Calkin-Wilf forest.
\item The rational number corresponding to $110001110000_2$ in the Calkin-Wilf tree is a descendant of the orphan root $110_2$. One can trace from the right, a sequence of four 0s, three 1s, two 0s, and then it offers neither two consecutive 0s nor three consecutive 1s.
\item The rational number corresponding to the position $110001110001_2$ in the Calkin-Wilf tree is an orphan in the $(2,3)$-Calkin-Wilf forest.
\end{itemize}

The following result formalizes the above criterion for an element to be an orphan or a child of a $(u,v)$-Calkin-Wilf tree.

\begin{proposition}\label{binrep}
Let $w$ be a vertex of a $(u,v)$-Calkin-Wilf tree, and $B(w)$ be the binary representation of its corresponding position in the original Calkin-Wilf tree.
\begin{enumerate}
\item[(a)] Suppose that $B(w) =  B_1\underbrace{0\ldots 0}_{i}\phantom{}_2$, i.e., the binary representation $B(w)$ ends in exactly $i$ 0s. If $i\geq u$, then $w$ is the left child of the vertex whose position is  $B_1\underbrace{0\ldots 0}_{i-u}\phantom{}_2$. Otherwise, $w$ is an orphan.  \\
\item[(b)] Suppose that $B(w) =  B_0\underbrace{1\ldots 1}_{j}\phantom{}_2$, i.e., the binary representation $B(w)$ ends in exactly $j$ 1s. If $j\geq v$, then $w$ is the right child of the vertex whose position is  $B_0\underbrace{1\ldots 1}_{j-v}\phantom{}_2$. Otherwise, $w$ is an orphan.  \\
\end{enumerate}
\end{proposition}

Another viewpoint for understanding the relationship between descendants in a $(u,v)$-Calkin-Wilf tree is via continued fractions. We begin the study of the relationship between continued fractions and $(u,v)$-Calkin-Wilf trees with the following useful lemma (see \cite{BBT} for the case $u=v=1$).

\begin{lemma}[Continued fraction relationship]\label{cf}
 Let $\frac{a}{b}$ be a positive rational number with continued fraction representation $\frac{a}{b}=[q_0,q_1,\dots,q_r]$. It follows that
 \begin{enumerate}
    \item[(a)] if $q_0=0$, then $\frac{a}{ua+b}=[0,u+q_1,\dots,q_r]$;
    \item[(b)] if $q_0\neq0$, then $\frac{a}{ua+b}=[0,u,q_0,q_1,\dots,q_r]$;
    \item[(c)] and $\frac{a+vb}{b}=[v+q_0,q_1,\dots,q_r]$.
 \end{enumerate}
\end{lemma}

\begin{proof}
Let $$\frac{a}{b}=q_0+\cfrac{1}{q_1+\ddots+\cfrac{1}{q_{r-1}+\cfrac{1}{q_r}}}.$$ Note that $\frac{a}{ua+b}=\big(u+\frac{b}{a}\big)^{-1}$, so
\begin{align}\label{cfu}
\frac{a}{ua+b} & =\cfrac{1}{u+\cfrac{1}{q_0+\cfrac{1}{q_1+\ddots+\cfrac{1}{q_{r-1}+\cfrac{1}{q_r}}}}}.
\end{align}
By considering the cases when $q_0=0$ and $q_0\neq 0$, we get $(a)$ and $(b)$. The remaining case follows from the fact that $\frac{a+vb}{b}=\frac{a}{b}+v$.
\end{proof}

Lemma~\ref{cf} shows that the continued fraction representations of rationals appearing in a $(u,v)$-Calkin-Wilf tree follow a nice pattern. In fact, in the case where $u=v=1$, we can recover several of the properties of the original Calkin-Wilf tree listed in Section \ref{intro}.

The next theorem gives more insight into the properties of coefficients in the continued fraction representation of rational numbers appearing in a $(u,v)$-Calkin-Wilf tree.

\begin{theorem}[Descendant conditions]\label{cfterms}
Suppose that $w$ and $w'$ are positive rational numbers with continued fraction representations $w=[q_0,q_1,\dots,q_r]$ and $w'=[p_0,p_1,\dots,p_s]$. Then $w'$ is a descendant of $w$ in the $(u,v)$-Calkin-Wilf tree with root $w$ if and only if the following conditions all hold:
\begin{enumerate}
\item[(a)] $s\geq r$ and $2\mid (s-r)$;
\item[(b)]for $0\leq j\leq s-r-1$, $v\mid p_j$ when $j$ is even and $u\mid p_j$ when $j$ is odd;
\item[(c)] for $2\leq i\leq r$, $p_{s-r+i}=q_i$;
\item[(d)] and
\begin{enumerate}
\item[(i)] if $q_0\neq 0$, then $p_{s-r}\geq q_0$, $v\mid(p_{s-r}- q_0)$ and $p_{s-r+1}=q_1$;
\item[(ii)] otherwise, if $q_0 = 0$, then $v\mid p_{s-r}$, $p_{s-r+1}\geq q_1$, and $u\mid(p_{s-r+1}- q_1)$.
\end{enumerate}
\end{enumerate}
\end{theorem}

\begin{proof}
$(\Rightarrow)$ We prove the first direction by induction. Note that $(a)$ holds by Lemma \ref{cf}, so our main concern will involve the remaining conditions.

Let $A_n$ be the set of descendants of $w$ of depth $n$. Then $A_1$ consists of both children of $w$. If $w'$ is the left child of $w$ and $q_0=0$, then, by Lemma \ref{cf}, $w'$ has a continued fraction representation $w'=[0,u+q_1,\dots,q_r]$. In this case, $s=r$, so $(b)$ is vacuously true and $(c)$ immediately holds. (Note that $(c)$ is also vacuously true if $r=1$.) Since $s-r=0$, it follows that $p_{s-r}=p_0=q_0$, which implies that $v\mid (p_{s-r}-q_0)$. Also, it is clear that $p_{s-r+1}\geq q_1$ and $u\mid (p_{s-r+1}-q_1)$ since $p_{s-r+1}=u+q_1$. This shows that part $(ii)$ of condition $(d)$ holds. The two remaining cases, where $w'$ is a left child of $w$ with $q_0\neq 0$ and where $w'$ is a right child of $w$, can be handled in a similar way using Lemma \ref{cf}. This shows that the theorem holds for $A_1$.

Now suppose that the desired result holds for $A_k$ for some $k\geq 1$ and assume that $w'\in A_{k+1}$. Furthermore, assume that $w'$ is the left child of some $w''\in A_k$, where $w''$ has a continued fraction representation $w''=[d_0,d_1,\dots,d_t]$. By Lemma \ref{cf}, if $d_0=0$, then $s=t$ and $w'=[0,u+d_1,\dots,d_t]$. Since $p_k=d_k$ for $0\leq k\leq t$ with $k\neq 1$, then, with the exception of one coefficient, the result holds. For the case $k=1$, notice that if $t>r$, then $u\mid d_1$, so $u\mid(u+d_1)$. If $t=r$, then $u+d_1-q_1>d_1-q_1\geq 0$ and $u\mid (d_1-q_1)$, so $u\mid(u+d_1-q_1)$. This implies the desired result.

As was the case with $A_1$, there are two remaining cases to handle. The proofs of the statement when $d_0\neq 0$ and when $w'$ is the right child of some $w''\in A_k$ are both similar to the argument presented above. We omit the details.

$(\Leftarrow)$ Using Lemma \ref{cf}, a simple computation shows that when $q_0\neq 0$, \begin{align}
w' & =R_v^{p_0/v} L_u^{p_1/u}\cdots R_v^{p_{s-r-2}/v} L_u^{p_{s-r-1}/u} R_v^{(p_{s-r}-q_0)/v}(w).\label{coeff}
\end{align} A similar formula gives the desired result when $q_0=0$.
\end{proof}

\begin{corollary}[Depth formula]\label{cfdepth}
Using the same hypothesis as Theorem \ref{cfterms}, if $n$ is the depth of $w'$, then
\begin{align}
n & = \frac{1}{v}\Bigg(\sum_{\substack{0\leq j\leq s-r-1\\j\text{ even}}}p_j + \sum_{\substack{0\leq i\leq r\\i \text{ even}}}(p_{s-r+i} - q_i)\Bigg) \label{depthform} \\
&\qquad + \frac{1}{u}\Bigg(\sum_{\substack{0\leq j\leq s-r-1\\j \text{ odd}}}p_j + \sum_{\substack{0\leq i\leq r\\i \text{ odd}}}(p_{s-r+i} - q_i)\Bigg).
\nonumber
\end{align}
\end{corollary}

The proof of Corollary \ref{cfdepth} follows from Theorem \ref{cfterms} by induction. Note that the majority of the terms in the sum \eqref{depthform} are actually zero. In the case where $u=v=1$, Corollary \ref{cfdepth} recovers the formula from Property~\ref{df}.

From Lemma \ref{cf} and Theorem \ref{cfterms}, we can construct a recursive algorithm that determines the orphan ancestor of $w'$ in the $(u,v)$-Calkin-Wilf that contains it. The algorithm makes heavy use of the continued fraction representation of $w'$.

\begin{algorithm}[ht!]
\caption{$(u,v)$-Calkin-Wilf tree orphan ancestor}\label{euclid}
\begin{algorithmic}[1]
\Procedure{ancestor}{$[p_0,p_1,\dots,p_s],u,v$}
\If{$s=0$}
    \If{$p_0 \leq v$}
        \textbf{return} $[p_0]$
    \Else
        \textbf{ return} \Call{ancestor}{$[p_0-v],u,v$}
    \EndIf
\ElsIf{$s=1$}
    \If{$0< p_0 < v$}
        \textbf{return} $[p_0,p_1]$
    \ElsIf{$p_0 > v$}
        \State \textbf{return} \Call{ancestor}{$[p_0-v,p_1],u,v$}
    \ElsIf{$p_0 = 0$ \textbf{and} $p_1 \leq u$}
        \textbf{return} $[0,p_1]$
    \Else
        \textbf{ return} \Call{ancestor}{$[0,p_1-u],u,v$}
    \EndIf
\Else
    \If{$p_0 < v$}
        \textbf{return} $[p_0,p_1,\dots,p_s]$
    \ElsIf{$p_0\geq v$}
        \State \textbf{return} \Call{ancestor}{$[p_0-v,p_1,\dots,p_s],u,v$}
    \ElsIf{$p_0 = 0$ \textbf{and} $0 < p_1 < u$}
        \State \textbf{ return} \Call{ancestor}{$[0,p_1,\dots,p_s],u,v$}
    \ElsIf{$p_0 = 0$ \textbf{and} $p_1 > u$}
        \State \textbf{ return} \Call{ancestor}{$[0,p_1 - u,\dots,p_s],u,v$}
    \Else
        \textbf{ return} \Call{ancestor}{$[p_2,\dots,p_s],u,v$}
    \EndIf
\EndIf
\EndProcedure
\end{algorithmic}
\end{algorithm}

For example, let $u=2$ and $v=3$. The continued fraction representation of $2147/620$ is given by $[3,2,6,4,5,2]$. Using the above algorithm, we can compute the list of ancestors of $2147/620$ as: $287/620=[0,2,6,4,5,2]$, $287/46=[6,4,5,2]$, $149/46=[3,4,5,2]$, $11/46=[0,4,5,2]$, $11/24=[0,2,5,2]$, $11/2=[5,2]$, and $5/2=[2,2]$. Since $1/2\leq5/2\leq3$, then $5/2$ is the orphan ancestor of $2147/620$.

By \eqref{coeff}, we see that the coefficients of the continued fraction of $2147/620$ encode the path taken from the orphan $5/2$ to the descendant $2147/620$. This can be computed as follows. Consider the continued fraction representation $[3,2,6,4,5,2]$ as a row vector. Extend the continued fraction representation of $5/2$ to a row vector of the same length by adding zeros at the front, $[0,0,0,0,2,2]$. Take the difference between both vectors, $[3,2,6,4,3,0]$. Divide the even-indexed (note that the leading terms is indexed by 0) terms by 3 and the odd-indexed terms by 2, $[1,1,2,2,1,0]$. Corollary \ref{cfdepth} states that the sum of the terms in this vector gives the depth of $2147/620$. The terms also show that $2147/620 = R_vL_uR^2_vL^2_uR_v(5/2)$. In particular, since \begin{equation}R_vL_uR^2_vL^2_uR_v = \begin{bmatrix}187 & 606\\54 & 175\end{bmatrix},\label{lem1example}\end{equation} then $a=187$, $b=606$, $c=54$, and $d=175$ in Lemma~\ref{replemma} for this case.

When $u=v=1$, the above discussion shows that every positive rational number appears in the original Calkin-Wilf tree (see~\cite{CW}).

\section*{Acknowledgments}

Support for this project was provided by a PSC-CUNY Award, jointly funded by The Professional Staff Congress and The City University of New York (\#67111-00 45 to the second author, and \#67136-00 45 to the third author).


\end{document}